\documentclass{amsart}

\usepackage{amsthm,amsmath,verbatim,amssymb}
\usepackage{color}
\newcommand{\CP}{\mathbb{CP}}

\newcommand{\kahler}{K\"ahler }

\newcommand{\szego}{Szeg\"o }
\newcommand{\wt}{\widetilde}
\renewcommand{\d}{\partial}
\newcommand{\ep}{\varepsilon}
\newcommand{\dd}{{\bf d}}

\newcommand{\dbar}{\bar\partial}

\newcommand{\R}{\mathbb{R}}
\newcommand{\E}{\mathbb{E}}

\newcommand{\la}{\lambda}
 \def   \half   {{\frac{1}{2}}}
\newcommand{\C}{\mathbb{C}}
\def \to {\rightarrow}
\newcommand{\ga}{\gamma}
\newcommand{\lcal}{\mathcal{L}}
\newcommand{\ccal}{\mathcal{C}}
\renewcommand{\H}{{\mathbf H}}
\newcommand{\hcal}{\mathcal{H}}
\newcommand{\mcal}{\mathcal{M}}
\newcommand{\ncal}{\mathcal{N}}
\newcommand{\ocal}{\mathcal{O}}
\newcommand{\scal}{\mathcal{S}}

\newtheorem{maintheo}{{\sc Theorem}}
\newtheorem{maincor}{{\sc Corollary}}

\newtheorem{theo}{{\sc Theorem}}[section]
\newtheorem{cor}[theo]{{\sc Corollary}}

\newtheorem{lem}[theo]{{\sc Lemma}}

\newenvironment{rem}{\medskip\noindent{\it Remark:\/} }{\medskip}
\newenvironment{defin-no-number}{\medskip\noindent{\it Definition:\/} }{\medskip}

\begin{document}\title[Supremum of $L^2$ normalized random holomorphic fields]{Median and mean of the Supremum of $L^2$ normalized random holmorphic fields}
\author[Renjie Feng]{Renjie Feng}
\author[Steve Zelditch]{Steve Zelditch}
\address{Department of Mathematics and Statistics, Mcgill University, Canada}
\email{renjie@math.mcgill.ca}
\address{Department of Mathematics, Northwestern University, USA}
\email{zelditch@math.northwestern.edu}

\thanks{Research partially supported by NSF grant  DMS-1206527.}

\date{\today}
\maketitle
\begin{abstract}  We  prove that the expected value and median of the supremum of $L^2$ normalized random holomorphic fields
of degree $n$ on $m$-dimensional K\"ahler manifolds are asymptotically of order $\sqrt{m\log n}$. There is an exponential concentration of measure
of the sup norm around this median value. Prior results  only gave the upper bound. The estimates are based
on the entropy methods of Dudley and Sudakov combined with a precise analysis of the relevant distance functions
and covering numbers using off-diagonal asymptotics of Bergman kernels. Recent work on the value distribution
are also used. 
\end{abstract}
\vspace{.5in}

The purpose of this note is to determine the asymptotic  mean and median  of the sup-norm   functionals
$$\lcal_{\infty}^n : H^0(M, L^n) \to \R_+, \;\;\;\; \lcal_{\infty}^n (s_n) = \sup_{z \in M} |s_n(z)|_{h^n} $$
on the subspace \begin{equation} \label{SH0} SH^0(M, L^n)  = \{s_n \in H^0(M, L^n): \|s_n\|_{h^n}^2: = \int_M |s_n(z)|^2_{h^n} d V = 1\}
\end{equation}  of $L^2$ normalized random holomorphic sections of the $n$th power $(L^n, h^n) \to (M, \omega)$
of a positive Hermitian holomorphic line bundle over a compact K\"ahler manifold. As discussed at length in \cite{SZ, SZ2} and \cite{FZ} (among many other articles), holomorphic  sections of positive line bundles are the analogues on compact complex manifold
of polynomials of degree $n$, and  in the special case of $M = \CP^m$ and $L= \ocal(1)$,  $H^0(M, L^n)$  is the space of homogeneous holomorphic polynomials of degree $n$ on $\C^{m+1}$ (see \S \ref{BACKGROUND} for background).
The inner product $||s_n||_{h^n}^2$ induces a unit mass  spherical Haar measure $\nu_n$ on  $SH^0(M, L^n)$ and we are interested
in the statistical properties of $\lcal_{\infty}^n$ in this  spherical ensemble. As discussed in \cite{FZ, SZ}, we regard the  spherical ensemble as primary since our goal is to measure  sup norms
of $L^2$ normalized sections. We denote the expectation of a random variable in any measure $\mu$ by $\E_{\mu}$,
and the  median by $\mcal_{\mu}$.

\begin{maintheo} \label{MAIN}  The median and mean value of $\lcal_{\infty}^n$ on $(SH^0(M, L^n), \nu_n)$ satisfy
$$ \mathcal M_{\nu_n}(\mathcal L^n_\infty) = \sqrt{m\log n}+o(\sqrt{\log n}), \;\; \mbox{resp.} \;\; \E_{\nu_n} \lcal_{\infty}^n
= \sqrt{m\log n}+o(\sqrt{\log n}) . $$

\end{maintheo}

The upper bound on the median with an unspecifed constant was proved  in \cite{SZ}, and the question of finding
its true order of magnitude was raised there.
To our knowledge, the lower bound is new, although there are many prior results  on suprema of
random processes  in other contexts, and the  proof   is based in part on   classical entropy methods of Dudley \cite{D,D2}  and Sudakov \cite{S}.  The main new ingredient is the analysis of the pseudo-metrics ${\bf d}_n$
induced by the holomorphic random fields, which makes use of the  off-diagonal asymptotics of the Szeg\"o kernel in
the complex geometric setting of \cite{SZ2}.  Use is also made of recent results on the value-distribution of
the fields \cite{FZ} in \S \ref{sharp} in getting precise bounds. The same methods apply in the real domain to random spherical harmonics
and their generalization to random Riemannian waves, except that the pseudo-metric in that setting is determined by
the spectral projections kernels for the Laplacian.

 The precise value of the median is needed to obtain a  concrete Levy concentration of measure result
for the sup norm. Levy concentration  (see \eqref{LEVY}) states that a Lipschitz functional is exponentially concentrated around its median value \cite{Le}. As noted in \cite{SZ}, the  functional $\lcal^n_{\infty}$ is Lipschitz with the estimate \eqref{levyformula} of Levy concentration of measure theorem,  Theorem \ref{MAIN} thus  has the following

\begin{maincor}  \label{LEVY2} There exists   constants $C, c>$ independent of $n$  so that
$$
 \nu_n \{ s_n \in SH^0(M, L^n) :  \left| \lcal^n_{\infty}(s_n) -  \sqrt{m\log n} \right|  \geq \epsilon \sqrt{\log n}  \} \leq
C n^{- c \epsilon^2}.  $$
for any $\epsilon>0$.
\end{maincor}

Throughout, we use $c$ and $C$ to denote positive constants which may differ in each instance.  The symbol $A\sim B$ means $A$ and $B$ are bounded from above an below by positive constants independent of $n$.

\subsection{\label{SKETCH} Sketch of the proof}

The  spherical measures $\nu_n$  are  asymptotically equivalent to   certain
 Gaussian measure which we call {\it normalized Gaussian measure} on $H^0(M, L^n)$ (see  \eqref{HGMg}-\eqref{law}).  We  first study the expectation and median of $\lcal^n_{\infty}$ of  normalized Gaussian random sections in Theorem \ref{normalized}; we then derive the result for $\nu_n$.

Let us recall the entropy estimates on suprema of Gaussian (or sub-Gaussian) random processes.
Let $(M, d)$ be a compact metric space.
 Given
a centered  random process $\{Y_x : x\in M\}$  (i.e. $\E Y_x = 0 $ for all $x \in M$), a pseudometric on $M$ may be
defined by
\begin{equation} \label{d}
{\bf d} (x,y) = \sqrt{\E |Y_x - Y_y|^2}. \end{equation}

The process $\{Y_x : x\in M\}$ is called \emph{sub-Gaussian} if
\begin{equation}\label{E:process}
\forall x,y\in M,\ \forall t>0, \quad
\mathbb P\big[|Y_x - Y_y| \ge t \big] \le 2 \exp\left[ -
  \frac{b\ t^2}{\dd^2(x,y)} \right]
\end{equation}
for some constant $b>0$. A Gaussian processes is sub-Gaussian.

Entropy estimates for suprema of (sub-) Gaussian processes involve the $\epsilon$-covering
number $N(M,\dd,\epsilon)$  of $(M,\dd)$, i.e.
 the minimal cardinality of an $\epsilon$-dense subset of $M$, i.e.
$$N(M,\dd,\epsilon) := \inf\{ \# \ncal: \ncal \subset M:
\forall x\in M \ \exists y\in \mathcal{N} : \ \dd(x,y)\le \epsilon\}. $$

Dudley's entropy upper  bound state:

\begin{maintheo}\label{T:Dudley} \cite{D,D2,K,MP,Li}
Let $\{Y_x : x\in M\}$ be a centered (sub-)Gaussian random process. Then
\begin{equation}
\E \sup_{x\in M} |Y_x| \le C \int_0^\infty \sqrt{\log N(M,\dd,\varepsilon)}
  \ d \epsilon,
\end{equation}
where $C>0$ depend only on the constant $b$ in the (sub-)Gaussian estimate
 for the process.
\end{maintheo}

If $Y_x$ is a centered Gaussian process, Sudakov's minoration gives the lower bound,
\begin{maintheo} \cite{S,Li}

\begin{equation}
\E \sup_{x\in M} |Y_x|  \geq  c \epsilon \sqrt{\log N(M, \dd, \epsilon)}
\end{equation}
\end{maintheo}


In Lemmas \ref{METRICSZ} and \ref{dis}, we  relate the distance $\dd_n$ for the $n$th normalized Gaussian process to
the  distance between points of the  Kodaira (or coherent states) embeddings
\begin{equation} \Phi_n: M \to  SH^0(M, L^n). \end{equation}
As proved in \cite{Ti,Ze}, this embedding is an asymptotically isometric embedding when properly normalized. Hence on very small length
scales, the intrinsic  K\"ahler distance and the extrinsic  $L^2$  distance  of $H^0(M, L^n)$  when restricted to $\Phi_n(M)$ are very similar.  The precise comparison is in Lemma \ref{INFDIS}.

We first prove the upper and lower bounds on expected sup norms for the {\it normalized  Gaussian ensemble} in Section \ref{proofmain} with unspecified constants; the case of the spherical ensemble  follows Lemma \ref{SPH}. Bounds of the median then follow  by the Levy concentration theorem (\S \ref{con}).  In the last section \S \ref{sharp}, we use the exact
asymptotic formula in \cite{FZ} for the value distribution to  estimate the constants in upper and lower bounds and show that the mean and the median are asymptotic to the sharp bound $\sqrt{m\log n}$.

\subsection{Prior results}

The study of sup-norms of Gaussian random fields has a long history, and we only indicate a few of the classical results.
In \cite{SaZy}, Salem-Zygmund studied sup norms
$$M_n(t): = \max_x |P_n(x, t)|$$ of random
 trigonometrical polynomials of the form,
$$P_n (x, t) = \sum_{m = 0}^n r_m \phi_m(t)  \cos(m x),$$
where $\{\phi_m\}$ is the orthonormal basis of Rademacher functions and
$R_n = \sum_{m = 1}^n r_m^2$. We recall that the Rademacher system is
the orthogonal system $\{\phi_m(t): = \mbox{sgn} (\sin 2^{m } \pi t) \}_{m=1}^{\infty}$ for $t \in [0, 1]$.
Let $M_n(t) = \max_x |P_n(x, t)|$.  In Theorems 4.3.1 resp.  4.5.1 of \cite{SaZy}, it is
proved that
$$
 c(\gamma)  \leq \liminf_{n\to\infty} \frac{M_n(t)}{(R_n \log n)^{\half}}  \leq  \limsup_{n\to\infty} \frac{M_n(t)}{(R_n \log n)^{\half}}
\leq A\;\;\; \mbox{almost surely}. $$
They also proved,
$$\mathbb P\{M_n(t)<C(R_n\log n)^{\half}\}\to 1$$
as $C$ or $n$ large enough.

Kahane \cite{K} gives an upper bound for sup-norms of  Gaussian random functions of the form
$$P(t_1, \dots, t_m) = \sum  a_j f_j(t_1, \dots, t_m)$$
where $\{f_j\}$ are complex trigonometric polynomials in $m$
variables of degrees less than or equal to $n$,  $a_j$ are normal random
variables and $\sum$ is a finite sum. In \cite{K},
(Chapter 6, Theorem 3),  it is proved that
$$\mathbb P\left\{\|P\|_{\infty} \geq\textstyle C (m\sum\|f_j\|_{\infty}^2
\log n \right)^{\half}\} \leq n^{-2} e^{-m}. $$


 In   \cite{SZ}  upper bounds   on the expected value of $\lcal^n_{\infty}$ in the
spherical ensembles of this article are proved:
$$  \nu_n\left\{s_n\in SH^0(M,L^n): \sup_M|s_n|_{h^n}>c\sqrt{\log
n}\right\} < O\left(\frac{1}{n^2}\right)\,, $$ for some constant
$c<+\infty$. (In fact, for any $k>0$, the probabilities are of order  $O(n^{-k})$ if one chooses $c$ to be sufficiently
large.)  It is also proved that  sequences of sections  $s_n \in SH^0(M, L^n)$ satisfy:
$$\|s_n\|_{\infty}= O(\sqrt{\log
n})\;\;\; \mbox{ almost surely}.$$
The proof is based on  the same
ingredients as the Dudley entropy bound of this note.

\section{\label{BACKGROUND} Background}

In this section, we go over the geometric background to our setting and the facts  about Szeg\"o kernels
that we need in the proofs of the main results.

\subsection{\kahler geometry}

The setting consists of a compact \kahler manifold  $(M,\omega)$  of complex dimension $m$
and a  positive Hermitian holomorphic line bundle  $(L,h)\to M$. We fix   a local non-vanishing holomorphic section $e$ of $L$ over an
open set $U \subset M$ such that locally $L|_U\cong U\times\C$. We define the K\"ahler potential
of $\omega$ by   $|e|_{h} = h(e, e)^{
1/2} = e^{- \phi}$.
  The curvature of the Hermitian metric $h$,
\begin{equation}\Theta_h=-\d\dbar \log |e|_h^2\end{equation}
is a positive $(1,1)$ form and $\omega = \frac{i}{2} \Theta_h$ \cite{GH}.

The Hermitian
metric $h$ induces a Hermitian metric $h
^n$ on the $n$th tensor power $L^n=L \otimes \cdots \otimes L$ of $L$,  given by $|e^{ \otimes n}|_{h^n}=|e|_h^n$.
In  local coordinate, we can write a global holomorphic section as $s_n=f_ne^{\otimes n}$ where $f_n$ is a holomorphic function on $U$, and  $|s_n|_{h^n}=|f_n|e^{-\frac {n\phi} 2}$.

 The spaces  $H^0(M,L^n)$  of global holomorphic sections of $L^n$ provide generalizations of polynomials of degree $n$
to $M$.
By the Riemann-Roch formula, the  dimension  $d_n = \dim H^0(M,L^n)$ grows at the rate
\begin{equation}\label{RR}
d_n=\frac{c_1(L)^m}{m!}n^m+O(n^{m-1})\,.\end{equation}


We  define an inner product on $H^0(M,L^n)$ by
\begin{equation}\label{innera}\langle s^n_1, s^n_2\rangle_{h^n}=\int_M h^n(s^n_1, s^n_2)dV,\,\,\,\, s^n_1, s^n_2\in H^0(M, L
^n)\end{equation}
 where $dV=\frac{\omega^m}{m!}$ is the volume form. We  assume the volume is normalized as $\int_M dV=1$.
In local coordinates,
\begin{equation}\label{innera2}\langle s^n_1, s^n_2\rangle_{h^n}=\int_Mf_1\bar f_2 e^{-n\phi}dV\end{equation}
where we write
$s_1^n=f_1^ne^{\otimes n}$ and $s_2^n=f_2^ne^{\otimes n}$ in the local coordinate.

We choose an orthonormal basis $\{s_j^n\}$ under this inner product and we can write every element in $H^0(M,L^n)$ as the orthogonal series
\begin{equation}\label{sections}s_n = \sum_{j = 1}^{d_n} a_j s^n_j,\end{equation}
where $s_j^n=f^n_je ^{\otimes n}$ in the local coordinate.
\subsection{\label{nun}Spherical ensemble}

We define the spherical probability measure   $d\nu_n $ to be  normalized  Haar measure on \begin{equation} SH^0(M, L^n) =\{s_n \in H^0(M, L^n) : ||s_n||_{L^2} = 1\}.
\end{equation}  We refer to the corresponding
probability space as the  spherical ensemble. Using the orhonormal basis $\{s_j^n\}$ we may identify
$SH^0(M, L^n)$ with the unit sphere $S^{2 d_n -1} \subset \C^{d_n}$.

\subsection{\label{GAUSS}  Gaussian measure}

We also endow $H^0(M, L^n)$ with   {\it normalized  Gaussian measures} adapted to the Hermitian metric and the associated inner product
\eqref{innera} on sections as follows: We put
\begin{equation} \label{HGMg} d\gamma_n(s_n)=(\frac{d_n}{\pi})^{d_n}e^
{-d_n|a|^2}da\,,\qquad s_n=\sum_{j=1}^{d_n}a^n_js^n_j\,,\end{equation}
where $\{s^n_1, \cdots, s^n_{d_n}\}$ is the orthonormal basis of  $H^0(M,L^n)$ with respect to the inner product \eqref{innera}.
 Equivalently, the
coefficients $a_j^n$ are complex Gaussian random variables which satisfy the following normalization conditions,
 \begin{equation}\label{law}\E a^n_k=0,\,\,\,\E a^n_k\bar a^n_j=\frac 1{d_n}\delta_{kj},\,\,\, \E a^n_ka^n_j=0\end{equation}
  Here, we denote  the expectation with respect to $\gamma_n$ by  $\E $.
 Under this normalization, we have the expected $L^2$ norm of $s_n$,  \begin{equation}\label{norm1}\E \|s_n\|^2_{h^n}= 1.\end{equation}

As proved in \cite{FZ}, this normalized Gaussian measure is asymptotically equivalent to the spherical measure $\nu_n$ \S \ref{nun} of principal concern
in this article.

\subsection{Lift to circle bundle $X_h$}  We identify the
sections \eqref{sections} with  the locally defined  functions  \begin{equation} \label{NEWSN} s_n=(\sum_{j=1}^{d_n}a^n_j f_j^n)e^{-\frac {n\phi}{2}} .\end{equation}
 This identification can be made global by  lifting holomorphic
sections   $s_n$ of $L^n$ to equivariant scalar  functions $\hat{s}_n: X_h \to \C$  on the unit circle
bundle $X_h \to M$ defined by the metric $h$ (see \cite{SZ2} for background). That is,  \begin{equation} \label{Xh} X_h =
\{v \in L^*: \|v\|_{h^*} =1\} \to M \end{equation} where  $\pi: L^* \to M$
denotes the dual line bundle to $L$ with dual metric $h^*$. We let
$A$ be the connection 1-form on $X$ given by the Chern $\nabla$; we
then have $dA =\pi^* \omega$, and thus $A$ is a contact
form on $X_h$, i.e., $A\wedge (dA)^m$ is a volume form on $X_h$.

We let $r_{\theta}x =e^{i\theta} x$ ($x\in X_h$) denote the $S^1$
action on $X_h$ and denote its infinitesimal generator by
$\frac{\partial}{\partial \theta}$. A section $s$ of $L$
determines an equivariant function $\hat{s}$ on $L^*$ by the rule
$\hat{s}(\lambda) = \left(\lambda, s(z) \right)$ ($\lambda \in
L^*_z, z \in M$).  We  restrict $\hat{s}$ to $X_h$ to obtain an equivariant
function transforming by  $\hat{s}(r_{\theta} x) = e^{i
\theta}\hat{s}(x)$. Similarly, a section $s_n$ of $L^{n}$
determines an equivariant function $\hat{s}_n$ on $X_h$: put
\begin{equation} \label{sNhat}\hat{s}_n(\lambda) = \left( \lambda^{\otimes
n}, s_n(z) \right)\,,\quad \la\in X_{h,z}\,,\end{equation} where
$\lambda^{\otimes n} = \lambda \otimes \cdots\otimes \lambda$;
then $\hat s_n(r_\theta x) = e^{in\theta} \hat s_n(x)$. We denote
by $\lcal^2_n(X_h)$ the space of such equivariant functions
transforming by the $n$-th character, and by ${\mathcal H}_n$ the
subspace of CR functions annihilated by the tangential
Cauchy-Riemann operator $\bar{\partial}_b.$ We refer to \cite{SZ2} for further
details and references.

The space ${\mathcal H}_n$ carries the natural inner product
$$\left\langle \hat{s}, \overline{\hat t} \right\rangle = \int_{X_h} \hat s\,
\overline{\hat t} \, dV_{X_h}, \;\;\;\; dV_{X_h} = A \wedge (d
A)^{m-1}. $$ Lifting the orthonormal basis to $\{\hat s_j^n\}$, we modify \eqref{sections}
to
write every element as
$$\hat{s}_n = \sum_{j = 1}^{d_n} a^n_j \hat s^n_j. $$

\medskip

We  trivialize the bundle $X_h \to M$ \eqref{Xh} using
a  {\it Heisenberg coordinate chart\/} at  $x_0 \in X_h$ , i.e. a coordinate chart
  of the
form
\begin{equation}\rho(z_1,\dots,z_m,\theta)= e^{i\theta} e^{- \phi}
e^*(z)\,,\label{coordinates}\end{equation} around $x_0$
 where
$(z_1,\dots,z_m)$ are preferred
coordinates centered at  $P_0=\pi(x_0)$ in the sense of \cite{SZ2}.
If  $s_n=fe^{\otimes
n}$ is a local section of $L^n$, then  by (\ref{sNhat}) and
(\ref{coordinates}),
\begin{equation}\label{sNhat*}\hat s_n(z,\theta) =
f(z) e^{-n \phi }e^{in\theta}\,.\end{equation}
Thus, \eqref{NEWSN} is  the lift $\hat{s}_n$ (with the common  factor of $e^{in\theta}$ supressed).

\subsection{\label{COMP} Comparison of $\gamma_n$ and $\nu_n$}

The  Gaussian measure $\gamma_n$  is asymptotically  concentrated near the unit sphere as $d_n \to \infty$,
so it may be expected that the median and mean of $\lcal^n_{\infty}$ should be asymptotically the
same in both ensembles.

The  spherical probability measure $\nu_d$ on the sphere $S^d(\sqrt{d})$
tends to the Gaussian measure  as $d \to \infty$ in the following sense: if $P_d:\R^d\to \R^k$ is the map,
$P_d(x)=\sqrt{d}(x_1,\dots,x_k)$, then for all $k$,  $P_{d*}\nu_d \to \ga_k = (2 \pi)^{-d/2} e^{-|x|^2/2} dx\,.$
Moreover,
\begin{equation} \label{COMPARE} \left\{ \begin{array}{l} \gamma_d\{x \in \R^d: ||x||^2 \geq \frac{d}{1 - \epsilon}  \} \leq e^{- \epsilon^2 d/4},\\  \gamma_d\{x \in \R^d: ||x||^2 \leq (1 - \epsilon) d \} \leq e^{- \epsilon^2 d/4}.
\end{array} \right. \end{equation}

To compare expectations of sup norms, we use the obvious

\begin{lem} \label{SPH}  The expected values of $\lcal_{\infty}^n$ with respect to  the spherical, resp. normalized
Gaussian measure, are related by
$$\begin{array}{lll} \E_{\gamma_n} \lcal_{\infty}^n &= &  C_n \E_{\nu_n}\lcal_{\infty}^n,  \end{array}$$
where
$$C_n = 1 + o(1), \;\; n \to \infty  . $$

\end{lem}

\begin{proof}

The one-parameter family of complex Gaussian measures on $H^0(M, L^n)$ may be written formally as
$$d\gamma_n^{\alpha} =(\frac{ \alpha}{\pi}) ^{d_n} e^{- \alpha ||s||^2} D s $$ where $Ds$ is Lebesgue measure. If we set $ \alpha=d_n$, then $d\gamma_n^\alpha$ becomes the normalized Gaussian ensemble $d\gamma_n$.

 For any $s \in H^0(M, L^n)$ and for any $r > 0$,
$\lcal_{\infty}^n ( r s) = \sup_M |r s(z)|_h = r \lcal^n_{\infty}(s)$. Hence,
$$\begin{array}{lll} \E_{\gamma^\alpha_n}  \lcal_{\infty}^n  &= &   \frac{\alpha^{d_n}}{\pi^{d_n}}\int_{H^0(M, L^n)}  \lcal_{\infty}^n(s)  e^{- \alpha  ||s||^2} D s
 \\ &&\\
&=&\frac{\alpha^{d_n}}{ \pi^{d_n}} \omega_{2 d_n} \int_0^{\infty} \int_{S H^0}     \lcal_{\infty}^n (\tilde s) r e^{- \alpha  r^2} r^{2d_n-1} dr d \nu_n\\ &&\\
&=& C_n \E_{\nu_n} \lcal^n_{\infty},
\end{array}$$
where we write $s=r\tilde s$ with $\tilde s\in SH^0(M,L^n)$, and
$$C_n = \frac{\alpha^{d_n}}{ \pi^{d_n}} \omega_{2 d_n} \int_0^{\infty}  r e^{- \alpha  r^2} r^{2d_n-1} dr
=  \frac{\alpha^{-\half}}{ 2\pi^{d_n}} \omega_{2 d_n}  \Gamma(d_n + \half). $$

We then put $\alpha = d_n$ to obtain
$$C_n = \frac{1}{2 d_n^{\half} \pi^{d_n}} \omega_{2 d_n}  \Gamma(d_n + \half).$$
Here, $\omega_k= \frac{2\pi^{k/2}}{\Gamma(k/2)}$ is the surface measure of the unit sphere $S^{k-1} \subset \R^{k}$.
Since $\frac{\Gamma(d_n + \half)}{\Gamma(d_n)} \sim d_n^{\half}$ we obtain that $C_n \simeq 1. $

\end{proof}

\subsection{Covariance kernel of $\gamma_n$ and Bergman-\szego kernels}\label{s-kernels }

The Bergman-Szeg\"o kernel $\Pi_n(x,y)$ is the Schwartz kernel of the  orthogonal projection
 \begin{equation}\Pi_n : \lcal^2(X_h) \rightarrow \hcal_n(X_h),
\end{equation}
   i.e.
\begin{equation} \Pi_n F(x) = \int_{X_h} \Pi_n(x,y) F(y) dV_{X_h} (y)\,,
\quad F\in\lcal^2(X_h)\,.
\end{equation} It is given in terms of the orthonormal basis by
\begin{equation}\label{szego}\Pi_n(x,y)=\sum_{j=1}^{d_n}
\hat s_j^n(x)\overline{ \hat s_j^n(y)}\,.\end{equation}
In local coordinates it has the form,
 \begin{equation}\label{szegokernel}
 \Pi_n(z,w)=\sum_{j=1}^{d_n} f^n_j(z)\overline{ f^n_j(w)}e^{-\frac{n(\phi(z)+\phi(w))}{2}} \end{equation}
It arises in probability as the  covariance kernel of $\gamma_n$, i.e.
\begin{equation} \label{SZ} \E( s_n(z) \overline{s_n(w)} ) = \frac{1}{d_n} \Pi_n(z,w).  \end{equation}

The Bergman-\szego kernels determine Kodaira maps
 $\Phi_n : M \to PH^0(M,L^n)'$  to projective
space, defined by \cite{GH}
\begin{equation}\label{Kmap} \Phi_n : M \to\CP^{d_n-1}\,,\qquad
\Phi_n(z)=\big[s^n_1(z):\dots:s^n_{d_n}(z)\big]\,.\end{equation}We lift the maps to $X_h$ by
\begin{equation}\label{lift}\wt{\Phi}_n : X_h \to
\C^{d_n}\,,\qquad
\wt\Phi_n(x)=(\hat s^n_1(x),\dots,\hat s^n_{d_n}(x))\,.\end{equation}
We observe that
\begin{equation}\label{PIPHI} \Pi_n(x,y)=\wt\Phi_n(x) \cdot
\overline{\wt\Phi_n(y)},\,\;\;\;
\Pi_n(x,x)=\|\wt\Phi_n(x)\|^2\,.\end{equation}

On the diagonal, the \szego kernel admits a complete asymptotic expansion \cite{Ze},
 \begin{equation} \label{CXDIAG}  \Pi_n(z,z)  =  a_0 n^m +
a_1(z) n^{m-1} + a_2(z) n^{m-2} + \dots \end{equation} for certain smooth
coefficients $a_j(z)$ with $a_0 = \pi^{-m}$.
This implies Tian's almost isometry theorem:
 Let $\omega_{FS}$ denote the Fubini-Study form on $\CP^{d_n-1}$.
Then \begin{equation} \label{TIAN} \|\frac{1}{n}  \Phi_n^*(\omega_{FS}) - \omega\|_{\ccal^k} =
O(\frac{1}{n}) \end{equation} for any $k$. We refer to \cite{SZ2} for notation and background.
We also need offf-diagonal asymptotics of the Bergman kernel  \cite{SZ2}.  We  denote by $r(z,w)$ the geodesic distance between $z,w$ with respect to the K\"ahler
metric $\omega$ on $M$.

\begin{theo}\label{SZFACTS}\begin{enumerate}
\item [a)] Within a  $\frac{C}{\sqrt{n}}$ neighborhood of the diagonal,
the Bergman-\szego kernel is given by
the scaling asymptotics:
\begin{equation}\begin{array}{l} n^{-m}  \Pi_N(z_0 + u/\sqrt{n}, \theta/n; z_0 + v/\sqrt{n}, 0)
\sim \Pi_1^{\H}
(u,\theta;  v, 0)\left[1 + O(1/\sqrt{n})\right].
\end{array}\end{equation}
Here
$$\Pi^\H_1(u,\theta;v,\psi)= \frac{1}{\pi^m} e^{i(\theta-\psi)+i\Im
(u\cdot \bar v)-\half |u-v|^2}\,$$ is the \szego kernel of the
reduced Heisenberg group.

\medskip
\item [b)] If  $r(z, w) \leq C/n^{1/3}$, we have:
\begin{equation}\label{neardiag2} |\Pi_n(z,w)| \le \left(\frac 1{\pi^m}
+o(1)\right) {n^m}\exp\left(-\frac {1-\ep} 2 n
d(z,w)^2\right)+O(n^{-\infty})\;. \end{equation}

\item [c)]  On all of $M$, we have:
\begin{equation}\label{offdiag} |\Pi_n(z, w)| \le Cn^m \exp\left(-\la \sqrt n\,
d(z,w) \right)\;.\end{equation}
for some positive $\lambda>0$.
\end{enumerate}

\end{theo}


The estimate (b) on the larger $n^{-1/3}$ balls is  from
\cite[Lemma~5.2(ii)]{SZ2}. The off-diagonal estimate (c) follows by an Agmon distance
argument, as noted by M. Christ
\cite{Ch}.

\section{The metrics $\dd_n$}
The proof of our main result is based on the Dudley's metric entropy method which relates the median of the suprema of a process by
its `pseudometric' . In this section, we will  compute the pseudometric (which turns out to be a metric) for our normalized Gaussian ensemble and the spherical ensemble. There is a clash of notation between the dimension
$d_n$ of $H^0(M, L^n)$ and the metric $\dd_n$, but both are standard and we distinguish them by putting the
metric in boldface. Recall that  $r(z,w)$ denotes the geodesic distance between $z,w$ with respect to the K\"ahler
metric $\omega$ on $M$.




\begin{lem}  \label{METRICSZ} In the the spherical $\nu_n$ and normalized Gaussian ensemble $\gamma_n$, we have

\begin{equation}\label{distance} \dd_n(z,w)=\frac1{\sqrt{d_n}}\sqrt {\Pi_n(z,z)+\Pi_n(w,w)-2\Re \Pi_n(z,w)} \end{equation}
where $\Pi_n(z,w)$ is the Szeg\"o kernel in \eqref{szegokernel}.

\end{lem}

\begin{proof}  We first consider $\gamma_n$,   the normalized Gaussian random sections \eqref{HGMg} (or the equivalent expression \eqref{NEWSN}).  By
 definition and by \eqref{law},
$$\begin{array}{lll} \dd^2_n(z,w)  & = & \E \left|(\sum_{j=1}^{d_n}a^n_j f_j^n(z))e^{-\frac {n\phi(z)}{2}}   - (\sum_{j=1}^{d_n}a^n_j f_j^n(w))e^{-\frac {n\phi(w)}{2}} \right|^2 \\&&\\
& = & \E (\sum_{j, k=1}^{d_n} a^n_j \bar{a}_k^n f_j^n(z)) \overline{f_j^n(z))} e^{-n\phi(z)}
+ \sum_{j, k=1}^{d_n} a^n_j \bar{a}_k^n f_j^n(w)) \overline{f_j^n(w))} e^{-n\phi(w)} ) \\&&\\&& - 2 \Re  \E   (\sum_{j, k=1}^{d_n}a^n_j
\bar{a}^n_k f_j^n(z) \overline{  f^n_j(w)} )e^{- \frac{n \phi(z)}{2} -\frac {n\phi(w)}{2}}  \\ &&\\
& = & \frac 1{d_n}(\Pi_n(z,z) + \Pi_n(w,w) - 2 \Re \Pi_n(z,w)). \end{array}$$

  We then observe that the expectations for $\nu_n$ are the same as in  \eqref{law}:
$$\nu_n(a_j \bar{a_k}) = \left\{ \begin{array}{ll} 0,  & k \not= j \\ \\
\frac{1}{d_n}, & k = j. \end{array} \right. $$
Indeed, for $j \not= k$, $a_j \bar{a_k}$  is a homogeneous harmonic polynomial of degree $2$ on $\C^{d_n} \simeq \R^{2d_n}$.
Indeed, if we write  $a_j = u_j + i v_j $ (and similarly for $a_k$ then
$a_j \bar{a_k} = (u_j u_k + v_k v_j) + i (u_j v_k - u_k v_j)$ and  $\Delta_{\R^{2n}} = \Delta_{u} + \Delta_v$ will
annihilate it.
Hence  its restriction to $S^{2d_n -1}$ is a spherical harmonic of degree $2$ and it is orthogonal to the constant function,
proving the first statement. For the second we use that $\E |a_j|^2$ is independent of $j$ and therefore equals its average.
It follows that $\dd_{\nu_n} = \dd_n$.

\end{proof}

We may interpret the distance in terms of the  lifted  Kodaira embeddings \eqref{lift}.

\begin{lem}\label{dis} $\dd_n(z, w) = \frac1{\sqrt{d_n}}||\tilde\Phi_n^z - \tilde\Phi_n^w||_{L^2}. $ Thus $\dd_n(z,w)$ is a metric on the \kahler manifolds. \end{lem}

\begin{proof}  By \eqref{PIPHI}  have,
$$\begin{array}{lll}  ||\tilde \Phi_n^z - \tilde\Phi_n^w||^2 & = & ||\tilde\Phi_n^z||^2 + ||\tilde\Phi_n^w||^2 - 2 \Re \langle \tilde\Phi_n^z,
\tilde\Phi_n^w \rangle. \end{array}$$
$d_n(z,w)$ is a metric since it satisfies the triangle inequality which is equivalent to the Minkowski inequality of the $L^2$-norm $\|\cdot\|_{L^2}$.

\end{proof}

Since  $d_n$
is asymptotically of order $n^{m}$ by \eqref{RR}, $\dd_n(z,w)$ is roughly  $n^{-m/2}$ times the distance
  in $ \C^{d_n}$  between $\tilde{\Phi}_n^z$ and $\tilde{\Phi}_n^w$.
The distance $\dd_n$  is globally very different from the Riemannian distance on $M$ defined
by the K\"ahler metric $\omega$.
However by \eqref{TIAN}, the Kodaira embeddings are almost isometric on tangent planes, hence  for distances of order $n^{-\half}$  they
nearly isometric. This is the key idea needed to calculate the covering numbers
$N( M, \dd_n,\epsilon)$, and then the metric entropies asymptotically.


The next Lemma gives the asymptotics of $\dd_n$ for separated $(z,w)$:
\begin{lem} \label{sep}  For all $z, w$,
\begin{equation} \dd_n(z,w) \leq  \sqrt 2. \end{equation}

 Moreover, for $z,w$ with $r(z,w)>\frac{c\log n}{\sqrt n}$ in the geodesic distance of $(M, \omega)$, then
\begin{equation}\label{rapid} \dd_n(z,w) =  \sqrt {\frac 2{d_n}(\Pi_n(z,z)+\Pi_n (w,w))}  + O(e^{- \sqrt{n} |z -w|} ) \simeq \sqrt{2} +  O(\frac{1}{\sqrt{n}})+ O(e^{- \sqrt{n} |z -w|}) \end{equation}
for sufficiently large $n$.

\end{lem}

\begin{rem}  The second statement can be interpreted as follows: $n^{-m/2} \tilde{\Phi}_n^z$ is almost
a unit vector, and $n^{-m/2} \tilde{\Phi}_n^z$  is almost orthogonal to $n^{-m/2} \tilde{\Phi}_n^w$  if
$r(z,w)  \geq \frac{c \log n}{\sqrt{n}}. $ .
\end{rem}

\begin{proof}By Lemma \ref{dis}, we have
$$\dd_n(z,w)=\frac1{d_n}\|\tilde\Phi_n^z -\tilde \Phi_n^w\|_{L^2}\leq \frac1{d_n}(\|\tilde\Phi_n^z\|_{L^2} + \|\tilde\Phi_n^w\|_{L^2})$$
The first  inequality   follows from the asymptotics of  $d_n$ \eqref{RR} and $\|\tilde\Phi_n^z\|$ \eqref{PIPHI}\eqref{CXDIAG}.
The inequality \eqref{rapid} is the consequence of Theorem \ref{SZFACTS}.
\end{proof}

We then have the asymptotics of the distance between very close points:

\begin{lem}\label{INFDIS} We have asymptotics,
\begin{equation} \label{small} \dd_n(z,w)  \simeq  \left\{ \begin{array}{ll} \sqrt {1-e^{- n r^2(z,w)}} , & r(z, w) < cn^{-\frac 12} \log n  \\ & \\
 n^{\half} r(z,w), & r(z,w) < cn^{-\frac 12-\eta}. \end{array} \right.
\end{equation} for any $\eta>0$.
Equivalently,
\begin{equation} \dd_n(z + \frac{u}{\sqrt{n}} , z + \frac{v}{\sqrt{n}}) \sim  |u - v|. \end{equation}
Here, $|u - v|$ is the Euclidean distance in normal coordinates.

\end{lem}

\begin{proof}
By  Theorem \ref{SZFACTS} with  $r(z,w) \leq \frac{c\log n}{\sqrt n}$,

$$\Pi_n(z,w)\sim e^{n(\phi(z,w)-\frac 12\phi(z)-\frac 12 \phi(w))}A_n(z,w)$$
where $$A_n=n^m(1+\frac 1n a_1+\cdots)$$ and  where $\phi(z,w)$ is the almost analytic extension of the \kahler potential
$\phi(z)$ (see \cite{SZ2} for background).
In particular, if $r(z, w) < cn^{-\frac 12-\eta}$,  it follows from Theorem \ref{SZFACTS} that
$$\begin{array} {lll} \dd_n(z,w)&=& \frac 1{\sqrt{d_n}}\sqrt {\Pi_n(z,z)+\Pi_n(w,w)-2\Re \Pi_n(z,w)}\\&&\\
&\sim& \sqrt {1-e^{- n r^2(z,w)}}\sim n^{\frac 12} r(z,w). \end{array}
$$

\end{proof}

\bigskip

\noindent{\bf Example: $(\CP^m, \omega_{FS})$ } In the case of the $SU(m + 1)$ ensemble where $M = \CP^m$
and $\omega  = \omega_{FS}$, the Fubini-Study metric,  the lifted \szego kernel on $S^{2m -1}$ is
\begin{equation}\label{szegosphere} \Pi_n(x,y)=   \frac{(n+m)!}
{\pi^m n!}\langle x, \bar{y}\rangle^n\,.\end{equation}
It is constant on the diagonal, equal to $\frac{1}{Vol(\CP^m)}$ times the dimension $d_n = \dim H^0(\CP^m, \ocal(n))$.
The lifted distance on $X_{h} = S^{2m-1}$ is then,
$$\dd_n(x,y) = \frac{1}{\sqrt{Vol(\CP^m)}} \sqrt{2 - 2 \Re \langle x, \bar{y} \rangle^n } =  \frac{\sqrt{2}}{\sqrt{Vol(\CP^m)}}
\sqrt{1 - \cos^n r(z, w) }.  $$
where the last equation holds when $x = (z, 0), y = (w, 0)$ (i.e. the angle in $S^1$ of the projection
$S^1 \to S^{2m -1} \to \CP^m$ is zero in local coordinates).  Since $\cos r = 1 - \frac{r^2}{2} + O(r^4)$,
$$\cos^n  r = e^{n \log (1 -  \frac{r^2}{2} + O(r^4))} = e^{-n  (\frac{r^2}{2} + O(r^4))}  = e^{- n \frac{r^2}{2}}(1 +
O(n r^4)),  $$
so the remainder term is negligible as long as $r \leq C n^{- \frac{1}{4} - \epsilon}. $ In this range,
$$  \dd_n(z,w )  \simeq  \frac{\sqrt{2}}{\sqrt{Vol(\CP^m)}}
\sqrt{1 -  e^{- n \frac{r^2(z,w)}{2}}}.  $$

\bigskip

\bigskip

Summarizing Lemmas \ref{sep} - \ref{INFDIS} (and assuming $Vol_{\omega}(M) = 1$),
 $$\dd_n\sim \sqrt 2\sqrt{1-e^{-nr^2/2}},\,\,\,r\leq c\frac{\log n}{\sqrt n}$$
or equivalently,
\begin{equation}\label{rexpression}r^2(z,w)\sim \frac 2n \log(1-\frac12 \dd_n^2)^{-1} ,\,\,\, \dd_n\in [0,\sqrt 2\sqrt{1-e^{-\frac {c^2}2(\log n)^2}}] \end{equation}

We denote $N(M, \omega, \epsilon')$ as the number of geodesic balls of radius $\epsilon'$ to cover the \kahler manifolds.  Then we have relation,
\begin{cor} \label{COVERING}     The covering number  $N(M, \dd_n, \epsilon)$ satisfies:
 \begin{equation} \label{NUB} N(M, \dd_n, \epsilon)= \left\{ \begin{array}{ll}  N(M, \omega,
\sqrt{\frac{2}{n}}  \sqrt{\log (1 - \half \epsilon^2)^{-1}}), &\epsilon \leq \sqrt{2} \; \sqrt{1 - \frac{1}{n}}, \\&\\

[1, N(M, \omega, \sqrt{\frac {2\log n}n})], &  \epsilon \geq \sqrt{2} \; \sqrt{1 - \frac{1}{n}}
\end{array} \right.
\end{equation}
\end{cor}

\begin{proof}
The first line is the direct consequence of the formula \eqref{rexpression}, the only thing we need to check is $\sqrt 2\sqrt{1-\frac1n}\in  [0,\sqrt 2\sqrt{1-e^{-\frac {c^2}2(\log n)^2}}]$, this is true as $n$ large enough.

For the second line, we know in Lemma \ref{dis} that $\dd_n$ is a metric, thus $N(M,\dd_n, \epsilon)$ will be a decreasing function with respect to the radius $\dd_n=\epsilon$, thus $N(M,\dd_n, \epsilon)$ will be bounded by the ends points $\epsilon=\infty$ and $\sqrt{2}\sqrt{1-\frac 1n}$. For $\epsilon=\sqrt{2}\sqrt{1-\frac 1n}$, the number of balls we need to covering the manifold will be $N(M, \omega, \sqrt{\frac{2\log n}n})$ by formula \eqref{rexpression}; for $\epsilon=\infty$, we know the diameter of the manifolds is $\dd_n(z,w)\leq\sqrt 2$ for all $z,w\in M$ in Lemma \ref{sep}, thus we only need $1$ ball to cover the manifold if $\epsilon\geq\sqrt 2$.

\end{proof}

\section{Upper bound and lower bound}\label{proofmain}

In this section we prove Theorem \ref{MAIN} except that we do not give sharp estimates on the coefficients
of $\sqrt{\log n}$. They are proved in the last section.
\subsection{Bound for mean}
The following extends the non-sharp version of  Theorem \ref{MAIN} to the normalized Gaussian ensemble as well as the spherical
ensemble:

\begin{theo}\label{normalized} Let $\E_n$ denote the expectation with respect to either  the spherical
ensemble $\nu_n$ or the normalized Gaussian ensemble $\gamma_n$. Then we have bounds,
\begin{equation}\label{ghg} c \sqrt{\log n} \leq  \E_n \mathcal L_\infty^n\leq C \sqrt{\log n}.
\end{equation}
where $c$ and $C$ can be chosen to be the same in both ensembles by Lemma \ref{SPH}.
\end{theo}

\subsubsection{Upper bound}
The upper will be given by the Dudley entropy bound of Theorem \ref{T:Dudley}, which holds for the
$\gamma_n$ and for $\nu_n$ since the latter is sub-Gaussian.
Since the metrics $\dd_n$ are the same in both ensembles, the same upper bound will hold.

By Lemma \ref{sep},
 $N(M,\dd_n, \epsilon)=1$ if $\epsilon> \sqrt{2}$, i.e., $\log N(M, \dd_n, \epsilon)=0$. Hence,
\begin{equation} \label{DUD2} \E_n\sup_M |s_n|_{h^n}\leq C\int_0^{
\sqrt{2}}\sqrt{\log N(M,\dd_n, \epsilon)}d\epsilon.  \end{equation}

We break up  the integral \eqref{DUD2} into two terms,
$$\int_0^{\sqrt {2}} =  \int_0^{\sqrt{2} \sqrt{1 - \frac{1}{n}}} + \int_{\sqrt{2} \sqrt{1 - \frac{1}{n}}} ^{\sqrt{2}}:=I_n+II_n.$$

In integral $I$, it follows from Corollary \ref{COVERING}, we have (for a constant $C_m > 0$ which depends only on
the dimension $m = \dim_{\C} M$
but which changes line to line),
\begin{equation} \begin{array}{lll} I_n  & \leq &  \int_0^{\sqrt{2} \sqrt{1 - \frac{1}{n}}}    \sqrt{\log N \left(M, \omega,
\sqrt{\frac{2}{n}}  \sqrt{\log (1 - \half \epsilon^2)^{-1}} \right) }d \epsilon \\ &&\\
&& \leq C_m  \;   \int_0^{\sqrt{2} \sqrt{1 - \frac{1}{n}}}  \sqrt{\log \left(\sqrt{\frac{2}{n}}  \sqrt{\log (1 - \half \epsilon^2)^{-1}} \right)^{-2m}} d \epsilon
\\ &&\\
&& \leq C_m\sqrt{\log \frac{n}{2}}\;   \int_0^{\sqrt{2} \sqrt{1 - \frac{1}{n}}}
 \sqrt{\left(1 - \frac{2}{  \log \frac{n}{2 } }
\log ( \sqrt{ \log (1 - \frac{\epsilon^2}{2})^{-1} )} \right)} d \epsilon
\\ &&\\
&& = C_m\sqrt{\log \frac{n}{2}}\;   \int_0^{\sqrt{2} \sqrt{1 - \frac{1}{n}}}  \sqrt{\left(1 - \frac{1}{  \log \frac{n}{2 } }
\log  \log (1 - \frac{\epsilon^2}{2})^{-1}  \right)} d \epsilon.  \end{array}\end{equation}
By dominated convergence,
$$   \int_0^{\sqrt{2} \sqrt{1 - \frac{1}{n}}}  \sqrt{\left(1 - \frac{1}{  \log \frac{n}{2 } }
\log  \log (1 - \frac{\epsilon^2}{2})^{-1}  \right)} d \epsilon \to \sqrt{2}. $$
Hence,
\begin{equation} \label{I} I_n = C_m  \sqrt{\log n}  (1 + o(1)).  \end{equation}
where $C_m$ depends only on the dimension.

On the other hand, by the second part of Corollary \ref{COVERING},
\begin{equation} \label{II} II_n \leq \sqrt{2\log N(M,\omega, \sqrt{\frac{2\log n}{n}})} \left(1 - \sqrt{1 - \frac{1}{n}} \right) << \sqrt{\log n}. \end{equation}

Combining \eqref{I}- \eqref{II} completes the proof.

 \subsubsection{Lower bound}
For the normalized Gaussian ensemble, the  lower bound is given by the Sudakov minoration principle,
\begin{equation} \label{SLB} \E_{\gamma_n}\sup_M |s_n|\geq c_m  \;\epsilon \sqrt{\log N(M, \dd_n, \epsilon)},\,\,\, \mbox{for all}\,\, \epsilon>0. \end{equation}
To obtain the lower bound it suffices to choose an optimal value of $\epsilon$. As in the calculation of Dudley's
integral, for $\epsilon\in[0,\sqrt2\sqrt{1-\frac 1n}]$, we have
$$\epsilon\sqrt{\log N(M, \dd_n, \epsilon)} \sim \epsilon\sqrt{\log \frac{n}{2}}\;    \sqrt{\left(1 - \frac{1}{  \log \frac{n}{2 } }
\log  \log (1 - \frac{\epsilon^2}{2})^{-1}  \right)}  \geq c \sqrt{\log n} $$
if we  choose some $\sqrt{2} > b > \epsilon = a > 0$ for $n$ large enough.

The lower bound for the spherical ensemble then follows by Lemma \ref{SPH}.

\subsection{Bounds for median}\label{con}

We now prove the non-sharp bounds for the median.

\begin{theo}\label{medianbounds}We have the following bounds for the median under the spherical ensemble $(SH^0(M, L^n),\nu_n)$,
\begin{equation}  c \sqrt{\log n} \leq  \mathcal M_{\nu_n}(\mathcal L^n_\infty)  \leq C \sqrt{\log n}.
\end{equation}
where the constants $c$ and $C$ can be chosen to be the same as the ones in \eqref{ghg}.
\end{theo}

The proof of Theorem \ref{medianbounds} and Corollary \ref{LEVY2} are based on the well-known Levy concentration of measure  theorem
\cite{Le} for Lipschitz continuous functions on spheres of large dimension.
 Let $\mcal(f)$ denote
the median of $f$. Then
\begin{equation} \label{LEVY} \mathbb P\left\{ x \in S^{d} : |f(x) - \mcal(f)| \geq r \right\} \leq
\exp \left(-\frac  {(d - 1) r^2 }{2 \|f\|_{Lip}^2}\right),
\end{equation}
where $$\|f\|_{Lip} = \sup_{d(x, y) > 0} \frac{|f(x) - f(y)|}{d(x,
y)|}
$$ is the Lipschitz norm.

We apply this result to $f = \lcal_{\infty}^n$. In  \cite{SZ} it is observed that
\medskip

\begin{enumerate}

\item [(i)] $\lcal_\infty^n$ is Lipschitz continuous with norm
$\frac{ n^{m/2}}{\sqrt{\log n}} \leq \|\lcal_\infty^n\|_{Lip} \leq
n^{m/2}$.

\medskip

\item [(ii)] The median of $\lcal_\infty^n$ satisfies:   $  \mathcal M_{\nu_n}(\mathcal L^n_\infty)
\leq C_m  \sqrt{\log n}$ for sufficiently large $n$.

\end{enumerate}
\medskip

The estimate of the  Lipschitz norm,
is based on the fact that the  $\lcal^2$-normalized
`coherent states' $\Phi_n^w(z) = \frac{\Pi_n(z,
w)}{\sqrt{\Pi_n(w,w)}}$ are the global maxima of $\lcal^\infty_n$
on $SH^0(M, L^n)$ and that  $\|\Phi_n^w(z)\|_{ \infty} =
\sqrt{\Pi_n(w,w)} \sim n^{m/2}$. It follows that
$$\big| \| s_1 + s_2\|_{\infty} - \|s_1\|_{\infty} \big| \leq 3 n^{m/2}.$$
Now let $s_1 $ have $L^\infty$ norm $\leq C \sqrt{\log n}$ and
let $s_1 = \Phi_n^w$ for some $w$. Then
$$\big| \| s_1 + s_2\|_{\infty} - \|s_1\|_{\infty} \big| \geq
\frac{ n^{m/2}}{\sqrt{\log n}}.$$

Summarizing the facts in our setting,  we can rewrite \eqref{LEVY} as
\begin{equation}\label{levyformula}
\mathbb P(|\mathcal L^n_\infty-   \mathcal M_{\nu_n}(\mathcal L^n_\infty)  |>r) \leq e^{-\frac {r^2}2}
\end{equation}
for $n$ large enough.
 Then the difference of the mean and median of sup norm is estimated to be,

 $$\begin{array}{lll} |  \mathcal M_{\nu_n}(\mathcal L^n_\infty)  - \E_{\nu_n}(\lcal_{\infty}^n)|
& \leq &   \E_{\nu_n} |\lcal^n_{\infty} -   \mathcal M_{\nu_n}(\mathcal L^n_\infty)  | \\&&\\ &=&\int_0^\infty\mathbb P(|\mathcal L^n_\infty-  \mathcal M_{\nu_n}(\mathcal L^n_\infty)  |>a)da\\ &&\\
&=& \int_0^{c\sqrt{\log n}}+\int_{c\sqrt{\log n}}^\infty \\&&\\
&\leq & c\sqrt{\log n}+\int_{c\sqrt{\log n}}^\infty e^{-\frac {a^2}2}da\\&&\\
&\leq & c\sqrt{\log n}+n^{-\frac{c^2}2}

 \end{array}$$ for any positive constant $c>0$.
This implies that the difference of median and mean is bounded by $c\sqrt{\log n}$ for small $c>0$.

\begin{cor}\label{differ} It follows that  $$\frac{1}{\sqrt{\log n}}\left(  \mathcal M_{\nu_n}(\mathcal L^n_\infty)  -\E_{\nu_n}(\lcal_{\infty}^n)\right)\to 0$$
as $n\to\infty$.

\end{cor}


Thus Theorem \ref{medianbounds} follows from Corollary \ref{differ} together with the bounds in Theorem \ref{normalized}. 

\section{Sharp bounds} \label{sharp}
In this section, we prove that the coefficients $c$ and $C$  in the  upper and lower bounds of Theorem \ref{normalized} and Theorem \ref{medianbounds}
in Section \ref{proofmain}
can be taken to be  $\sqrt m$.

\subsection{Upper bounds using the value density}\label{reprove}

We first prove the sharp upper bound in Theorem \ref{MAIN}  using the
results of \cite{FZ} on  the expected distribution of critical values of the spherical random sections.


 We define the normalized empirical measure of the critical values of random sections by, \begin{equation} CV_n(s_n) =\frac 1{n^m} \sum_{z: \nabla_n s_n=0}\delta_{|s_n|_{h^n}}\end{equation}
where $s_n \in SH^0(M, L^n)$ and $\nabla_n$ is the Chern connection of the line bundle $L^n$ with respect to the Hermitian metric $h^n$ \cite{GH}.  
The  expected density of critical values for the spherical ensemble was determined in \cite{FZ} to by given by
\begin{equation}\lim_{n\to \infty}\E_{\nu_n}  CV_n =p(x)e^{-x^2}\end{equation} in the sense of distribution, where $p(x)$ is a smooth function with polynomial growth  (Theorem 1 and Theorem 2 in \cite{FZ} ). In fact, $p(x)\sim c_m x^{m(m+1)+1}$ where $c_m$ is a universal  constant only depending on the dimension and independent of $n$.

We define the random variable,
\begin{equation}X_a = \langle \sum_{\nabla_n s_n=0} \delta_{|s_n|_{h^n}}, 1_{[a,\infty)} \rangle \end{equation}
which is the number of critical values that fall in the interval $[a,\infty)$.
Then,\begin{equation}\label{che1}\mathbb P( \sup_M |s_n|_{h^n}\geq a)=\mathbb P(  X_a\geq 1) \end{equation}
By Chebyshev's inequality, we have,
\begin{equation}\label{che2}\mathbb P(  X_a \geq 1)\leq\E X_a\sim n^m\int_a^\infty p(x)e^{-x^2}dx\end{equation}
for $n$ large enough.

Recall for any nonnegative random variable $X$, we have the identity,
\begin{equation}\label{formulaexpectation}\E X=\int_0^\infty \mathbb P(X>a)da\end{equation}
Letting  $X  \sup_M |s_n|_{h^n}$ gives
\begin{equation}\label{time}
\E_{\nu_n}\sup_M |s_n|_{h^n}= \int_0^\infty \mathbb P(\sup_M |s_n|_{h^n}>a)da=\int_0^{c\sqrt{\log n}}+\int_{c\sqrt{\log n}}^\infty=:I+II
\end{equation}
The first term is bounded by $c\sqrt{\log n}$ since  the probability is always less than $1$.
For the second term, we apply formulas \eqref{che1}\eqref{che2}, by choosing suitable constant $c$,  as $n$ large enough we will have,
$$II\leq  n^m\int_{c\sqrt{\log n}}^\infty \int_a^\infty p(x)e^{-x^2}dxda\sim n^m\int_{c\sqrt{\log n}}^\infty p_1(a)e^{-a^2}da \leq Cn^{-k}$$
for some constant $C>0$ and $k>0$, where in the second inequality we use the integration by part several times and $p_1(a)$ is a smooth function with polynomial growth.  The upper bound in  Theorem \ref{normalized} follows from these estimates of $I$ and $II$.
To obtain the optimal $C = C_m$  depending only on the
dimension we consider the minimum value $C_m$  of $C$ so that
$$\mathbb P(\sup_M|s_n|_{h^n}>C \sqrt{\log n})\leq  \frac 12. $$
Setting $a=C\sqrt{\log n}$ in \eqref{che2}, we get for sufficiently large $n$,
$$ \begin{array}{lll} \mathbb P(\sup_M|s_n|_{h^n}>C \sqrt{\log n}) & \leq &   c n^m\int_{C\sqrt{\log n}}^\infty x^{m(m+1)+1}e^{-x^2}dx \\ && \\  & \leq & c n^m (C\sqrt{\log n})^{m(m+1)}e^{-C^2\log n}\\&&\\
&\leq & \frac 12, \;\; \mbox{ as long as}\;\; C \geq  \sqrt{m+\frac {m(m+1)}{2}\frac{\log \log n}{\log n}}.
\end{array} $$

It follows that
$$\limsup_{n \to \infty} \frac{  \mathcal M_{\nu_n}(\mathcal L^n_\infty)   }{\sqrt{\log n}} \leq \sqrt{m}$$
and by Corollary \ref{differ}, we have, \begin{equation}\label{firstbound}\limsup_{n\to \infty} \frac{\E_{\nu_n}\mathcal L^n_\infty}{\sqrt{\log n}}\leq \sqrt m.\end{equation}


\subsection{Lower bound}

The lower bound of Theorem \ref{MAIN} for the mean follows
from a precise analysis of the constant in Sudakov's minoration \eqref{SLB} for the normalized
Gaussian case.

There is a universal estimate of the constant appearing in \eqref{SLB}. We follow \cite{Li}, Lemma 10.2.  In Lemma 10.2, the numerical constant $c_*=0.64$ is chosen such that the
inequality (10.8) is true for any integer $n$, but as stated in the proof of Lemma 10.2, the constant c can be chosen to be
any  $c<\sqrt 2$ if we have infinite many points. If we combine this with Theorem 10.5 in \cite{Li}, we have
$$\mathbb E\sup_T X \geq \epsilon \sqrt{\log N(T, \epsilon)}$$
in Sudakov's minoration.

Our Gaussian random fields are complex valued and are therefore equivalent to a  real two dimensional Gaussian process. Write $X=Y+iZ$ for two real standard independent Gaussian processes. Then we have,
 $$\sup |X|=\sup \sqrt{Y^2+Z^2}\geq \frac1{\sqrt 2}\sup |Y+Z|\geq \frac1{\sqrt 2} \sup (Y+Z)$$
 Then we apply Sudakov's minoration to the real process $Y+Z$ to get the lower bound,
 $$\mathbb E \sup |X|\geq \frac \epsilon{\sqrt 2}  \sqrt{\log {N(T, \epsilon)}} $$
 where the $L^2$ metric is given by $d(t,s)=\sqrt{\mathbb E(Y_t+Z_t-Y_s-Z_s)^2}=\sqrt{\mathbb E |X_t-X_s|^2}$.

In our case, we will have,
\begin{equation} \label{CSUD} \E_{\gamma_n}  (\sup_M|s_n|_{h^n})\geq \frac{\epsilon}{\sqrt 2}\sqrt{\log N(M, \dd_n, \epsilon)}
\end{equation}
for any $\epsilon>0.$
In  the proof of the lower bound we found that for $\epsilon=\sqrt 2\sqrt{1-\frac1n}$,
$N(M, \dd_n, \epsilon)=N(M, \omega, \sqrt{\frac{2\log n}{n}})=(\frac{2\log n}{n})^{-m}$ ( Corollary \ref{COVERING}), so by \eqref{CSUD} we have
$$\mathbb E_{\gamma_n}\sup_M|s_n|_{h^n}\geq \sqrt{m\log n}$$
for  $n$ large enough. It follows that, for the  normalized Gaussian sections in Theorem \ref{normalized}, we have

$$\liminf_{n\to \infty} \frac{\E_{\gamma_n}\mathcal L^n_\infty}{\sqrt{\log n}}\geq \sqrt m
$$
The  lower bounds   for the spherical measures $\nu_n$ then follow from  Lemma \ref{SPH},
\begin{equation}\label{lowerbound2} \liminf_{n\to \infty} \frac{\E_{\nu_n}\mathcal L^n_\infty}{\sqrt{\log n}}\geq \sqrt m
\end{equation}
Thus we have the sharp estimate of the mean in Theorem \ref{MAIN} if we combine \eqref{firstbound}\eqref{lowerbound2}.
The sharp estimate for  the median in Theorem \ref{MAIN}  follows from Corollary \ref{differ}.


\end{document}